\documentclass[12pt]{amsart}
\usepackage{amssymb}
\usepackage{amsfonts}
\usepackage{latexsym}
\usepackage{amscd}
\usepackage{amsmath,delarray}
\usepackage[mathscr]{euscript}
\usepackage{xy} \xyoption{all}
\usepackage[active]{srcltx}

\usepackage{tikz}

\usepackage{color,graphics}

\newcommand{\ol}{\overline}

\newcommand{\uloopr}[1]{\ar@'{@+{[0,0]+(-4,5)}@+{[0,0]+(0,10)}@+{[0,0] +(4,5)}}^{#1}}
\newcommand{\uloopd}[1]{\ar@'{@+{[0,0]+(5,4)}@+{[0,0]+(10,0)}@+{[0,0]+ (5,-4)}}^{#1}}
\newcommand{\dloopr}[1]{\ar@'{@+{[0,0]+(-4,-5)}@+{[0,0]+(0,-10)}@+{[0, 0]+(4,-5)}}_{#1}}
\newcommand{\dloopd}[1]{\ar@'{@+{[0,0]+(-5,4)}@+{[0,0]+(-10,0)}@+{[0,0 ]+(-5,-4)}}_{#1}}
\newcommand{\luloop}[1]{\ar@'{@+{[0,0]+(-8,2)}@+{[0,0]+(-10,10)}@+{[0, 0]+(2,2)}}^{#1}}

\newcommand{\dotedge}{\ar@{.}}
\newcommand{\eqedge}{\ar@{=}}

\theoremstyle{plain}
\newtheorem{theorem}{Theorem}[section]
\newtheorem{lemma}[theorem]{Lemma}
\newtheorem{proposition}[theorem]{Proposition}
\newtheorem{corollary}[theorem]{Corollary}

\theoremstyle{definition}
\newtheorem{definition}[theorem]{Definition}
\newtheorem{example}[theorem]{Example}

\newtheorem{remark}[theorem]{Remark}

\newtheorem*{remark*}{Remark}
\newtheorem*{remarks*}{Remarks}

\newtheorem*{assumption*}{Assumption}

\newtheorem{notation}[theorem]{Notation}

\numberwithin{equation}{section}

\begin{document}

\title{Finitely presented simple modules over Leavitt path algebras}
\author{Pere Ara}
\address{Departament de Matem\`atiques, Facultat de Ci\`encies, Edifici C, Universitat Aut\`onoma de Barcelona,
08193 Bellaterra (Barcelona), Spain.} \email{para@mat.uab.cat}
\author{Kulumani M. Rangaswamy}
\address{Department of Mathematics, University of Colorado at Colorado Springs, Colorado Springs, Colorado 
80918, USA}
\email{krangasw@uccs.edu}
\thanks{The first-named author was supported by DGI
MICIIN-FEDER MTM2011-28992-C02-01, and by the Comissionat per Universitats i
Recerca de la Generalitat de Catalunya.}
\subjclass[2010]{Primary 16D60; Secondary 16D70} \keywords{Leavitt path
algebra, simple module, primitive ideal, finitely presented module.}

\maketitle


\begin{abstract}
Let $E$ be an arbitrary graph and $K$ be any field. We construct various
classes of non-isomorphic simple modules over the Leavitt path algebra
$L_{K}(E)$ induced by vertices which are infinite emiters, closed paths which
are exclusive cycles and paths which are infinite,\ and call these simple
modules Chen modules. It is shown that every primitive ideal of $L_{K}(E)$ can
be realized as the annihilator ideal of some Chen module. 
Our main result establishes the equivalence between a graph theoretic condition and 
various conditions concerning the structure of simple modules over $L_K(E)$.
\end{abstract}

\section{Introduction}

Leavitt path algebras were introduced and initially studied in
\cite{AA}, \cite{AMP}, as algebraic analogues of graph C$^{\ast}$-algebras
and as natural generalizations of the Leavitt algebras of type $(1,n)$ built in \cite{Lea}.
The study of the module theory over Leavitt path algebras was initiated in 
\cite{AB}, in connection with some questions in algebraic K-theory.  
Very recently, following the results of \cite{GR}, Chen \cite{C}
has provided a method of constructing simple modules $V_{[p]}$ over a Leavitt path
algebra $L_{K}(E)$ of an arbitrary graph $E$ by using the equivalence class
$[p]$ of infinite paths tail-equivalent to a fixed
infinite path $p$ in $E$. (See below for the definition of tail-equivalence 
for infinite paths.)
He also constructed simple modules $\mathbf{N}_{w}$
corresponding to various sinks $w$ in $E$. The authors have used in \cite{AR} this family
of simple modules to determine the structure of the Leavitt path algebras of arbitrary graphs
which have only a finite number of isoclasses of simple modules.
It is as well interesting to observe that Chen's work is related 
to some constructions in non-commutative algebraic 
geometry (see \cite{Smithtiling} and \cite{Smithadvs}).

In this paper, we introduce
additional classes of simple modules using vertices which are infinite
emitters and also exclusive cycles, and call all these simple modules over
$L_{K}(E)$ \textit{Chen modules}. We give a description of the annihilators of
the various Chen modules and, as a consequence, we show that every primitive
ideal of $L_{K}(E)$ can be realized as the annihilator of a Chen module.
Here we take advantage of results from \cite{Ranga}, describing the structure 
of the primitive ideals over a Leavitt path algebra of an arbitrary graph (see also
\cite{ABR}, where the structure of primitive Leavitt path algebras of arbitrary graphs 
is described.) We also show, using results from \cite{C}, that the Chen modules are pairwise non-isomorphic.

Next we investigate the structure of simple modules over a Leavitt path algebra of a {\it finite}
graph $E$. The structure of the simple finitely presented $L_K(E)$-modules was determined in \cite{AB}
in terms of the irreducible finite-dimensional representations of the usual path algebra of the reverse
graph $\ol{E}$ of $E$. A lot is known about these representations, for instance Le Bruyn and Procesi 
determined in \cite[Section 5]{LP} the possible dimension vectors for them. So, a natural question is 
to determine all the finite graphs $E$ such that all simple 
$L_K(E)$-modules are finitely presented. One can also ask what are the connections between finitely presented simple 
modules and Chen simple modules. To begin
with, we show that the Chen module $V_{[p]}$ corresponding to an infinite path
$p$ is finitely presented if and only if $p$ is tail-equivalent to the
rational infinite path $ggg\cdot\cdot\cdot$ where $g$ is some closed
path.  

For an algebra $A$, 
denote by $\widehat{A}$ the set of isoclasses of simple left $A$-modules, and by ${\rm Prim}(A)$ the set of primitive ideals of $A$.
There is a canonical map
$$\widehat{A}\longrightarrow {\rm Prim} (A)$$
sending  $[N]$ to $\text{Ann}_A(N)$, the annihilator of $N$.

\medskip

Our main result is the following:

\begin{theorem}
 \label{theor:all-characterizations}
 Let $E$ be a finite graph and $K$ an arbitrary field. Write $L=L_K(E)$. 
 Then the following conditions are equivalent:
 \begin{enumerate}
  \item  Every simple left $L$-module is finitely presented.
\item  Every simple Chen module is finitely presented.
  \item  Every vertex $v$ in $E$ is the base of at most one cycle.
\item The map $\widehat{L}\to {\rm Prim}(L)$ is a bijection.
\item All simple left $L$-modules are Chen modules.
 \end{enumerate}
 \end{theorem}

It is interesting to note, from the recent investigation done in \cite{AAJZ}, that
the algebras $L_{K}(E)$ appearing in the theorem are precisely the Leavitt path algebras having
finite Gelfand-Kirillov dimension.

\section{Preliminaries}

Before we proceed to set the basic definitions, let us remark that there is a 
lack of uniformity in the notation and terminology used in graph theory.
Even in the setting of Leavitt path algebras or graph C*-algebras, different authors often
use  different conventions regarding the basic concepts from graph theory.
In particular we advise the reader that the notation herewith will be the same as 
in \cite{AAS} but will differ substantially from the one used by Chen in \cite{C}.

A (directed) graph $E=(E^{0},E^{1},r,s)$ consists of two sets $E^{0}$ and
$E^{1}$ together with maps $r,s:E^{1}\rightarrow E^{0}$. The elements of
$E^{0}$ are called \textit{vertices} and the elements of $E^{1}$
\textit{edges}. We generally follow the notation, terminology and results from
\cite{AAS}. We outline some of the concepts and results that will be used in
this paper.

A vertex $v$ is called a \textit{sink} if it emits no edges, that is,
$s^{-1}(v)=\emptyset$, the empty set. The vertex $v$ is called a
\textit{regular vertex} if $s^{-1}(v)$ is finite and non-empty and $v$ is
called an \textit{infinite emitter} if $s^{-1}(v)$ is infinite. For each $e\in
E^{1}$, we call $e^{\ast}$ a ghost edge. We let $r(e^{\ast})$ denote $s(e)$,
and we let $s(e^{\ast})$ denote $r(e)$. A \textit{finite path} $\mu$ of length
$n>0$ is a finite sequence of edges $\mu=e_{1}e_{2}\cdot\cdot\cdot e_{n}$ with
$r(e_{i})=s(e_{i+1})$ for all $i=1,\cdot\cdot\cdot,n-1$. In this case
$\mu^{\ast}=e_{n}^{\ast}\cdot\cdot\cdot e_{2}^{\ast}e_{1}^{\ast}$ is the
corresponding ghost path. The set of all vertices on the path $\mu$ is denoted
by $\mu^{0}$. Any vertex $v$ is considered a path of length $0$.

Given an arbitrary graph $E$ and a field $K$, the \textit{Leavitt path }%
$K$\textit{-algebra }$L_{K}(E)$ is defined to be the $K$-algebra generated by
a set $\{v:v\in E^{0}\}$ of pairwise orthogonal idempotents together with a
set of variables $\{e,e^{\ast}:e\in E^{1}\}$ which satisfy the following conditions:

(1) \ $s(e)e=e=er(e)$ for all $e\in E^{1}$.

(2) $r(e)e^{\ast}=e^{\ast}=e^{\ast}s(e)$\ for all $e\in E^{1}$.

(3) (The ``CK-1 relations") For all $e,f\in E^{1}$, $e^{\ast}e=r(e)$ and
$e^{\ast}f=0$ if $e\neq f$.

(4) (The ``CK-2 relations") For every regular vertex $v\in E^{0}$,
\[
v=\sum_{e\in E^{1},s(e)=v}ee^{\ast}.
\]

A non-trivial path $\mu$ $=e_{1}\dots e_{n}$ in $E$ is \textit{closed} if $r(e_{n}%
)=s(e_{1})$, in which case $\mu$ is said to be based at the vertex $s(e_{1})$.
A closed path $\mu$ as above is called \textit{simple} provided it does not
pass through its base more than once, i.e., $s(e_{i})\neq s(e_{1})$ for all
$i=2,...,n$. The closed path $\mu$ is called a \textit{cycle based at} $v$ 
if $s(e_1)=v$ and it does not
pass through any of its vertices twice, that is, if $s(e_{i})\neq s(e_{j})$
for every $i\neq j$. 

When talking about a cycle we loosely understand the set of closed paths 
obtained by rotation of a representative of $c$. When talking about a cycle based at a vertex $v$ we
understand the closed path in this family which starts at $v$.

A cycle $c$ is called an \textit{exclusive cycle} if it
is disjoint with every other cycle; equivalently, no
vertex on $c$ is the base of a different cycle other than the rotate of $c$ based at $v$.
These cycles were termed  cycles without (K) in \cite{Ranga}.

An \textit{exit }for a path $\mu=e_{1}\dots e_{n}$ is an edge $e$ such that
$s(e)=s(e_{i})$ for some $i$ and $e\neq e_{i}$. A graph $E$ is said to satisfy
\textit{Condition (L)} if every cycle in $E$ has an exit.

If there is a path from vertex $u$ to a vertex $v$, we write $u\geq v$. A
subset $D$ of vertices is said to be \textit{downward directed }\ if for any
$u,v\in D$, there exists a $w\in D$ such that $u\geq w$ and $v\geq w$. A
subset $H$ of $E^{0}$ is called \textit{hereditary} if, whenever $v\in H$ and
$w\in E^{0}$ satisfy $v\geq w$, then $w\in H$. A hereditary set is
\textit{saturated} if, for any regular vertex $v$, $r(s^{-1}(v))\subseteq H$
implies $v\in H$.

A subset $S$ of $E^{0}$ is said to have the \textit{Countable Separation
Property} (CSP) with respect to a set $C$, if $C$ is a countable subset of
$E^{0}$ with the property that for each $u\in S$ there is a $v\in C$ such that
$u\geq v$.

We shall be using the following concepts and results from \cite{T} in our
investigation. A \textit{breaking vertex }of a hereditary saturated subset $H$
is an infinite emitter $w\in E^{0}\backslash H$ with the property that
$1\leq|s^{-1}(w)\cap r^{-1}(E^{0}\backslash H)|<\infty$. The set of all
breaking vertices of $H$ is denoted by $B_{H}$. For any $v\in B_{H}$, $v^{H}$
denotes the element $v-\sum_{s(e)=v,r(e)\notin H}ee^{\ast}$. Given a
hereditary saturated subset $H$ and a subset $S\subseteq B_{H}$, $(H,S)$ is
called an \textit{admissible pair.} The admissible pairs form a partially
ordered set under the relation $(H_{1},S_{1})\leq(H_{2},S_{2})$ if and only if
$H_{1}\subseteq H_{2}$ and $S_{1}\subseteq H_{2}\cup S_{2}$. Given an
admissible pair $(H,S)$, $I(H,S)$ denotes the ideal generated by
$H\cup\{v^{H}:v\in S\}$. It was shown in \cite{T} that the graded ideals of
$L_{K}(E)$ are precisely the ideals of the form $I(H,S)$ for some
admissibile pair $(H,S)$. Moreover, $L_{K}(E)/I(H,S)\cong L_{K}%
(E\backslash(H,S))$. Here $E\backslash(H,S)$ is the \textit{quotient graph of
}$E$ in which\textit{ }$(E\backslash(H,S))^{0}=(E^{0}\backslash H)\cup
\{v^{\prime}:v\in B_{H}\backslash S\}$ and $(E\backslash(H,S))^{1}=\{e\in
E^{1}:r(e)\notin H\}\cup\{e^{\prime}:e\in E^{1},r(e)\in B_{H}\backslash S\}$
and $r,s$ are extended to $(E\backslash(H,S))^{1}$ by setting $s(e^{\prime
})=s(e)$ and $r(e^{\prime})=r(e)^{\prime}$. Thus when $S=B_{H}$,
$E\backslash(H,B_{H})^{0}=E^{0}\backslash H$ and $E\backslash(H,B_{H}%
)^{1}=\{e\in E^{1}:r(e)\notin H\}$, so we can identify the graph $E\backslash
(H,B_{H})$ with the subgraph $E/H$ of $E$.

If $H$ is a hereditary saturated subset of $E^0$, $c$ is a cycle without exits
in $E\backslash(H, B_H)=E/H$, based at $v\in E^0\setminus H$, and $f(x)$ is a polynomial in $K[x,x^{-1}]$, 
we will denote by  $I(H,B_H, f(c))$ the ideal of $L_K(E)$ generated by $I(H,B_H)$ and $f(c)$. Here $f(c)$ is the element 
of $L_K(E)$ obtained by formally substituting $x$ by $c$, $x^{-1}$ by $c^*$ 
and the constant term $a_0$  by  $a_0v$ in the canonical expression of $f(x)$ as a polynomial in $x,x^{-1}$.

\section{Special types of simple modules}

Let $E$ be an arbitrary graph and $L=L_{K}(E)$. In this section, we introduce
three new classes of simple left modules over $L$ as an addition to the two
types of simple left modules $\mathbf{N}_{w}$ and $V_{[p]}$ defined by X.W.
Chen (\cite{C}) and described below. We call all these simple modules Chen
modules. We first give a description of the annihilators of these Chen modules.
Then we show that every primitive ideal of $L$ can be realized as the
annihilator ideal of some Chen module. All the Chen modules are shown to be
pairwise non-isomorphic. These results are utilized in the next section.

Given an infinite path $p=e_{1}e_{2}\cdot\cdot\cdot e_{n}\cdot\cdot\cdot$
\ and an integer $n\geq1$, Chen (\cite{C}) defines $\tau_{\leq n}%
(p)=e_{1}\cdot\cdot\cdot e_{n}$ and $\tau_{>n}(p)=e_{n+1}e_{n+2}\cdot
\cdot\cdot$ . Two infinite paths $p,q$ are said to be \textit{tail-equivalent}
if there exist positive integers $m,n$ \ such that $\tau_{>m}(p)=\tau_{>n}%
(q)$. This is an equivalence relation and the equivalence class of all paths
tail equivalent to an infinite path $p$ is denoted by $[p]$. An infinite path
$p$ is called a \textit{rational path} if $p=ggg\cdot\cdot\cdot$ where $g$ is
some (finite) closed path in $E$.

Given an infinite path $p$, Chen defines $V_{[p]}=%
{\textstyle\bigoplus\limits_{q\in\lbrack p]}}
Kq$, a $K$-vector space having $\{q:q\in\lbrack p]\}$ as a basis. $V_{[p]}$ is
made a left $L$-module by defining, for all $q\in\lbrack p]$ and all $v\in
E^{0}$, $e\in E^{1}$,

$v\cdot q=q$ or $0$ according as $v=s(q)$ or not;

$e\cdot q=eq$ or $0$ according as $r(e)=s(q)$ or not;

$e^{\ast}\cdot q=\tau_{>1}(q)$ or $0$ according as $q=eq^{\prime}$ or not.

In \cite{C}, Chen shows that under the above action of $L$, $V_{[p]}$ becomes
a simple left $L$-module.

Similarly, given a sink $w$ in the graph $E$, Chen defines $\mathbf{N}_{w}$ to
be the $K$-vector space having as its basis all the (finite) paths in $E$ that end in
$w$. By defining an action of $L$ on the basis elements of $\mathbf{N}_{w}$ in
the same fashion as was done for $V_{[p]}$ (with the addition that $e^{\ast
}\cdot w=0$ for all $e\in E^{1}$), he shows that $\mathbf{N}_{w}$ becomes a simple
left $L$-module.

Throughout this section, we shall use the following notation.

For $v\in E^{0}$, define
\[
M(v)=\{w\in E^{0}:w\geq v\}\text{ and }H(v)=E^{0}\setminus M(v).
\]
Similarly, if $p$ is an infinite path in $E$, we define
\[
M(p)=\{w\in E^{0}:w\geq v\text{ for some }v\in p^{0}\}\quad \text{ and }\quad 
H(p)=E^{0}\setminus M(p).
\]

Clearly both $M(v)$ and $M(p)$ are downward directed sets. Also, for any
vertex $v$ which is a sink or an infinite emitter, and for any infinite path
$p$, the sets $H(v)$ and $H(p)$ are hereditary saturated subsets of $E^{0}$.
If $v$ is a finite emitter, it might be that $H(v)$ is not saturated, and 
$v$ may belong to the saturation of $H(v)$. Note also that $H(p)=H([p])$, that
is, $H(p)$ does not depend on the tail-representative of $[p]$.

\begin{lemma}
\label{lem:annih-sink} If $w$ is a sink, then the annihilator of the simple
module $\mathbf{N}_{w}$ is $I(H(w),B_{H(w)})$.
\end{lemma}

\begin{proof}
Denote by $J$ the annihilator of $\mathbf{N}_{w}$. We first prove that
$I(H(w),B_{H(w)})\subseteq J$. Since $J$ is an ideal of $L(E)$, it suffices
to check that $H(w)$ and $\{v^{H(w)}:v\in B_{H(w)}\}$ annihilate
$\mathbf{N}_{w}$. If $v\in H(w)$ then $v$ does not connect to $w$ and thus
$vp=0$ for every path $p$ ending at $w$. If $v\in B_{H(w)}\subset
E^{0}\backslash H(w)$, let $p$ be a path in $E$ such that $s(p)=v$ and
$r(p)=w$. Let $e$ be the initial edge of $p$, say $p=ep_{1}$. Then $r(e)\notin
H(w)$, because $r(e)$ connects to $w$. It follows that
\[
v^{H(w)}p=(v-\sum_{f\in s^{-1}(v),r(f)\notin H}ff^{\ast})ep_{1}=(e-e)p_{1}=0.
\]
Therefore, we have shown the inclusion $I(H(w),B_{H(w)})\subseteq J$. In
order to show the reverse inclusion, consider the graph $F= E\backslash(H(w),B_{H(w)})$, 
and recall that $L(E)/I(H(w),B_{H(w)})\cong
L(F)$. We can look at $\mathbf{N}_{w}$ as a simple $L(F)$-module, and we have
to show that it is faithful as a $L(F)$-module. Write $\overline{J}$ for the
annihilator of $\mathbf{N}_{w}$ as a left $L(F)$-module. Obviously we have
$\overline{J}\cap F^{0}=\emptyset$, because every vertex in $F$ connects to
$w$. On the other hand, $F$ satisfies condition (L) as every vertex in $F$
connects to the sink $w$ and this means, by Proposition 1 of \cite{AMMS},
that every non-zero ideal of $L(F)$ contains a vertex. Consequently, $\overline
{J}=0$. This proves the result.
\end{proof}

Recall that a cycle $c$ is said to be an \textit{exclusive cycle} if no vertex
on $c$ is the base of a different cycle (other than the rotate of $c$ based at that vertex).

If an infinite path $p$ is tail-equivalent to the rational path $c^{\infty}$,
where $c$ is a cycle in $E$, we say that $p$ \textit{ends in a cycle}.

\begin{lemma}
\label{lem:annih-infpath} Let $p$ be an infinite path. Then

(1) \ If $p$ does not end in an exclusive cycle then the annihilator of
$V_{[p]}$ is $I(H(p),B_{H(p)})$.

(2) \ If $p$ ends in an exclusive cycle $c$ based at a vertex $v$, then the
annihilator of $V_{[p]}$ is
$I(H(p),B_{H(p)},(c-v))$.
\end{lemma}

\begin{proof}
(1) The proof is similar to the proof of Lemma \ref{lem:annih-sink}. We leave
the details to the reader.

(2) Proceeding as in the proof of Lemma \ref{lem:annih-sink}, we arrive at a
simple module $V_{[p]}$ over $L(F)$, where $F=E\backslash (H(p),B_{H(p)})$, and $F$ has a
unique cycle without exits, which is $c$. Moreover $F^{0}\cap\overline{J}=0$,
where $\overline{J}$ is the annihilator of $V_{[p]}$ in $L(F)$ and hence a
primitive ideal of $L(F)$. By Theorem 4-3 (iii) in \cite{Ranga}, there exists
an irreducible polynomial $f$ in $K[x,x^{-1}]$ such that $\overline{J}$ is the
ideal generated by $f(c)$. Since $c-v$ annihilates $V_{[p]}$, we
conclude that $f=x-1$.
\end{proof}

Next we wish to introduce new classes of simple modules, similar to the simple
modules $N_{w}$ and $V_{[p]}$.

Let $[q]$ be an infinite rational path with $q=c^{\infty}$ and $c=e_{1}%
e_{2}\cdots e_{n}$ a cycle based at $v$. In \cite{C} Chen defines, for $a\in
K^{\times}$, a certain simple $L(E)$ module $V_{[q]}^{a}$, which is the
twisted module $V_{[q]}^{\sigma}$, where $\sigma$ is the \textquotedblleft
gauge\textquotedblright\ automorphism of $L(E)$ sending $v$ to $v$ for $v\in
E^{0}$, $e$ to $e$ and $e^{\ast}$ to $e^{\ast}$ for $e\in E^{1}$ with $e\neq
e_{1}$, and $e_{1}$ to $ae_{1}$ and $e_{1}^{\ast}$ to $a^{-1}e_{1}^{\ast}$.
Denoting by $\ast$ the module operation in $V_{[q]}^{a}$, we have $c\ast
q=\sigma(c)c^{\infty}=acc^{\infty}=ac^{\infty}=aq$ and similarly $c_{i}\ast
c_{i}^{\infty}=ac_{i}^{\infty}$ for all rotates $c_{i}$ of $c$. Moreover,
$V_{[q]}^{a}$ is a simple module.

As a slight modification of the above construction, let $f(x)=1+a_{1}%
x+\cdots+a_{n}x^{n}$, $n\geq1$, be an irreducible polynomial in $K[x,x^{-1}]$
and let $c=e_{1}e_{2}\cdots e_{m}$ be an exclusive cycle. Set $q=c^{\infty}$.
We are going to define a new module $V_{[q]}^{f}$. Let $K^{\prime}%
=K[x,x^{-1}]/(f(x))$, which is a field because $f(x)$ is irreducible. Define a
$L_{K^{\prime}}(E)$-module by $M=V_{[q]}^{\overline{x}}$. Observe that $M$ is
well-defined because $\overline{x}$ is invertible in $K^{\prime}$, and that
$M$ is a simple $L_{K^{\prime}}(E)$-module. We denote by $V_{[q]}^{f}$ the
$L_{K}(E)$-module obtained by restricting scalars on $M$ from $L_{K^{\prime}%
}(E)$ to $L_{K}(E)$.

\begin{lemma}
\label{lem:also-simple-over-K} The $L_{K}(E)$-module $V_{[q]}^{f}$ is simple.
\end{lemma}

\begin{proof}
Let $U$ be a nonzero $L_{K}(E)$-submodule of $V_{[q]}^{f}$. We first claim that, 
given two nonzero scalars
$\lambda,\mu$ in $K^{\prime}$ such that  
$\lambda q\in U$, then also $\mu q\in U$. For this it clearly suffices to
prove that $(\overline{x}\lambda)q\in U$. 
We have
\[
c\ast (\lambda q)  =\sigma(c)(\lambda
c^{\infty})=(\overline{x}\lambda)q.
\]
which shows the claim. 
Next, we follow the proof of \cite[Theorem 3.3(1)]{C}. Let $u=\sum _{i=1}^l \lambda_i p_i$ be a nonzero element in $U$,
with $\lambda_i \in K'\setminus \{0\}$ and $p_i$ distinct paths in $[q]$. 
We can uniquely write
$p_i=p_{i}'c^{\infty}$, where $p_{i}'$ is a finite path (possibly of length $0$)
which does not involve $e_{1}$. We can assume that the length of $p_1'$ is larger than or equal to 
the maximum of the lengths of all the other paths $p_i'$, $i\ge 2$.   We get
$$ [(p_1')^*]\ast u = \lambda _1 c^{\infty}$$
because the path $p_1'$ does not involve the edge $e_1$ and has maximum length. 
Therefore $\lambda_1 q\in U$. Now let $\mu$ be an arbitrary nonzero element of $K'$, and $p_0\in [q]$.
We have $p_0= p_0'q$, with $p_0'$ a finite path not involving $e_1$. 
By the claim we have $\mu q\in U$, and so
$$\mu p_0= p_0' \ast (\mu q ) \in U.$$
It follows that $U=V_{[q]}^{f}$, showing the result.
\end{proof}

The module $V_{[q]}^{f}$ above has the appropriate annihilator, as follows.

\begin{lemma}
\label{lem:annih-f(x)} Let $f(x)=1+a_{1}x+\cdots+a_{n}x^{n}$, $n\geq1$, be an
irreducible polynomial in $K[x,x^{-1}]$ and let $c=e_{1}e_{2}\cdots e_{m}$ be
an exclusive cycle. Set $q=c^{\infty}$. Then the annihilator of $V_{[q]}^{f}$
is $I(H(q),B_{H(q)},f(c))$.
\end{lemma}

\begin{proof}
As in the proof of Lemma \ref{lem:annih-infpath}(2), we only have to show that
the annihilator $\overline{J}$ of $V^{f}_{[q]}$ in $L(F)$ contains $f(c)$,
where $F=E\backslash (H(q), B_{H(q)})$. But this follows from
\[
f(c)\ast c^{\infty} = \sigma(f(c)) c^{\infty}= f(\sigma(c))c^{\infty}=
f(\overline{x}) c^{\infty} = 0.
\]

\end{proof}

Let $v$ be an infinite emitter such that $v\in B_{H(v)}$.

Then we can build the primitive ideal $P=I(H(v),B_{H(v)}\setminus\{v\})$
(see \cite{Ranga}) and the factor ring
\[
L(E)/P\cong L(F)\,
\]
where $F=E\backslash (H(v),B_{H(v)}\setminus\{v\})$. Here $F^{0}=(E^{0}\setminus
H(v))\cup\{v^{\prime}\}$,
\[
F^{1}=\{e\in E^{1}:r(e)\notin H(v)\}\cup\{e^{\prime}:e\in E^{1},r(e)=v\}
\]
and $r$ and $s$ are extended to $F$ by setting $s(e^{\prime})=s(e)$ and
$r(e^{\prime})=v^{\prime}$ for all $e\in E^{1}$ with $r(e)=v$. Note that
$v^{\prime}$ is a sink in $F$. We claim that $M_{F}(v^{\prime})=F^{0}$. Since
$M_{E}(v)=E^{0}\setminus H(v)$, it suffices to show that $v$ connects to
$v^{\prime}$. Now, since $v\in B_{H(v)}$, there is $e\in E^{1}$ such that
$s(e)=v$ and $r(e)\in M_{E}(v)$. If $r(e)=v$ then $e^{\prime}\in F^{1}$ and
$s(e^{\prime})=v$, $r(e^{\prime})=v^{\prime}$. If $r(e)\neq v$, then, since
every vertex in $M_{E}(v)$ connects to $v$, there exists a path $p=f_{1}\cdots
f_{m}$ in $E$ such that $s(p)=r(e)$ and $r(p)=v$. Now $q:=ef_{1}f_{2}\cdots
f_{m-1}f_{m}^{\prime}$ is a path in $F$ such that $s(q)=v$ and $r(q)=v^{\prime
}$, as claimed.

Accordingly we may consider the simple module $\mathbf{N}_{v^{\prime}}$ of
$L(F)$ introduced by Chen corresponding to the sink $v^{\prime}$ in $F$. Now
$\mathbf{N}_{v^{\prime}}$ is a simple faithful $L(F)$-module by Lemma
\ref{lem:annih-sink}, because $M_{F}(v^{\prime})=F^{0}$. Using the quotient
map $L(E)\rightarrow L(F)$, we may view $\mathbf{N}_{v^{\prime}}$ as a simple
module over $L(E)$. We shall denote this $L(E)$-module by $\mathbf{N}%
_{v}^{B_{H(v)}}$.

\begin{lemma}
\label{lem:annih-inf-emitter} Assume that $v$ is an infinite emitter and that
$v\in B_{H(v)}$. Then the annihilator of $\mathbf{N}_{v}^{B_{H(v)}}$ is
precisely $I(H(v),B_{H(v)}\setminus\{v\})$.
\end{lemma}

\begin{proof}
This follows from the fact that the simple $L(F)$-module $\mathbf{N}%
_{v^{\prime}}$ is faithful.
\end{proof}

Next, suppose $v$ is an infinite emitter such that $r(s^{-1}(v))\subseteq
H(v)$.

Then $v$ is the unique sink of $E\backslash (H(v),B_{H(v)})$. Let $\mathbf{N}%
_{v}$ be the corresponding simple $L_{K}(E\backslash(H(v),B_{H(v)}))$-module
introduced by Chen. It is clear that $\mathbf{N}_{v}$ is a faithful simple
$L_{K}(E\backslash(H(v),B_{H(v)}))$-module. Consider $\mathbf{N}_{v}$ as a
simple $L_{K}(E)$-module through the quotient map $L_{K}(E)\rightarrow
L_{K}(E\backslash(H(v),B_{H(v)}))$. We denote this as $\mathbf{N}_{v}^{H(v)}$.
The next lemma follows immediately.

\begin{lemma}
\label{lem:sink-in-the-quotient} Let $v$ be an infinite emitter and suppose
that $r(s^{-1}(v))\subseteq H(v)$. Then the annihilator of the simple module
$\mathbf{N}_{v}^{H(v)}$ is $I(H(v),B_{H(v)})$.
\end{lemma}

\begin{definition}
\label{def:Chen-module} Let $E$ be an arbitrary graph and $K$ an arbitrary
field. By a \textit{Chen module} we mean a simple left $L_{K}(E)$-module of
one of the following types:

\begin{enumerate}
\item $\mathbf{N}_{w}$, where $w$ is a sink in $E$;

\item $\mathbf{N}_{v}^{B_{H(v)}}$, where $v$ is an infinite emitter such that
$v\in B_{H(v)}$;

\item $\mathbf{N}_{v}^{H(v)}$, where $v$ is an infinite emitter such that
$r(s^{-1}(v))\subseteq H(v)$;

\item $V_{[p]}$, where $p$ is an infinite path on $E$;

\item $V_{[q]}^{f}$, where $q=c^{\infty}$, $c$ is an exclusive cycle, and
$f(x)$ is an irreducible polynomial in $K[x,x^{-1}]$, with $f(x) \ne x-1$.
\end{enumerate}
\end{definition}

\begin{proposition}
 \label{prop:pairwise-non-iso}
 All simple modules listed in Definition \ref{def:Chen-module} are pairwise non-isomorphic. 
 \end{proposition}

 \begin{proof}
  This follows from the computation of the annihilators of these modules and from \cite[Theorems 3.3(2), 3.7(3)]{C}.
   \end{proof}

We can now state the main result of this section.

\begin{theorem}
\label{thm:realizingprims} Let $E$ be an arbitrary graph and $K$ an arbitrary
field, and let $P$ be any primitive ideal of $L_{K}(E)$. Then there exists a
Chen simple module $S$ such that the annihilator of $S$ is $P$.
\end{theorem}

\begin{proof}
Let $P$ be a primitive ideal of $L(E)$, and set $H=P\cap E^{0}$. By
\cite[Theorem 4.3]{Ranga}, $P$ satisfies one of the following:

\begin{enumerate}
\item[(i)] $P=I(H,B_{H},f(c))$, where $c$ is a exclusive cycle based at a
vertex $u$, $E^{0}\setminus H=M(u)$, and $f(x)$ is an irreducible polynomial
in $K[x,x^{-1}]$.

\item[(ii)] $P$ is a graded ideal of the form $I(H,B_{H}\setminus\{u\})$,
where $u\in B_{H}$, and $M(u)=E^{0}\setminus H$;

\item[(iii)] $P$ is a graded ideal of the form $I(H,B_{H})$, and
$E\backslash(H,B_{H})$ is downward directed, satisfies the condition (L) and the
countable separation property.
\end{enumerate}

Suppose that (i) holds. Write $q=c^{\infty}$. Then $H=H(q)=E^{0}\setminus
M(u)$ and so, by Lemma \ref{lem:annih-f(x)}, the annihilator of $V_{[q]}^{f}$
is precisely $I(H,B_{H},f(c))=P$.

Next assume that (ii) holds. Now $u$ is an infinite emitter with $u\in
B_{H(u)}$ and, since $M(u)=E^{0}\setminus H$, we must have $H(u)=H$. Thus the
annihilator of $\mathbf{N}_{u}^{B_{H(u)}}$ is $P=I(H,B_{H}\setminus\{u\})$,
by Lemma \ref{lem:annih-inf-emitter}.

Finally, suppose that (iii) holds. Let $S$ be a countable (finite or infinite)
subset of $E^{0}\setminus H$ such that every vertex of $E\backslash(H,B_{H})$
connects to some vertex of $S$. We claim that either there is a (unique) sink in
$E\backslash(H,B_{H})$ to which all the vertices in $E\backslash(H,B_{H})$
connect, or else there exists an infinite path $p$ on $E\backslash (H,B_{H})$
such that each vertex in $E\backslash (H,B_{H})$ connects to a vertex in $p^{0}$. 
To see this, set $F:=E\backslash(H,B_{H})$.

If $F$ has a sink $v$, then since $F^{0}$ is downward directed, $v$ is a
unique sink, $v\in S$ and $F^{0}=M(v)$.

Assume that $F$ does not have any sink. If $S$ is infinite, let $v_{1}
,v_{2},v_{3},\dots$ be an enumeration of the elements of $S$. If
$S=\{v_{1},\dots,v_{k}\}$ is finite, we consider the infinite sequence
$v_{1},\dots,v_{k},v_{1},\dots,v_{k},v_{1},\dots,v_{k},\dots$ obtained by
repeating the finite sequence $v_{1},\dots,v_{k}$ infinitely many times, and
we write $v_{kr+i}=v_{i}$. Proceeding as in the proof of Theorem 3.5 in
\cite{ABR}, we may inductively define a sequence $\lambda_{1},\lambda
_{2},\dots,$ of paths in $F$ such that

\begin{enumerate}
\item $\lambda_{i}$ is an initial segment of $\lambda_{j}$ whenever $i\le j$,

\item The length of $\lambda_{i}$ is $\ge i$ for all $i$, and

\item $v_{i} \ge r(\lambda_{i})$ for all $i$.
\end{enumerate}

Now, we may use the paths $\lambda_{i}$ to build an infinite path $p$ such
that each vertex of $S$ connects to a vertex in $p^{0}$.

Therefore every vertex of $F$ connects to a vertex in $p^{0}$, as claimed.

We also note that when the path $p$ constructed above is a rational infinite
path $ggg\cdot\cdot\cdot$, where $g$ is a closed path, Condition (L) on $F$
together with the condition that every vertex connects to $g^0$ imply 
that $g$ cannot be an exclusive cycle.

If there is a unique sink $v$ in $E\backslash(H,B_{H})$ to which all the
vertices in $E\backslash(H,B_{H})$ connect, then $H=H(v)$, and there are two
possibilities: either $v$ is a sink in $E$, or else $v$ is an infinite emitter
such that $r(s^{-1}(v))\subseteq H(v)$. In either case, by Lemmas
\ref{lem:annih-sink} and \ref{lem:sink-in-the-quotient}, the annihilator of
$\mathbf{N}_{v}$, respectively of $\mathbf{N}_{v}^{H(v)}$, is precisely
$P=I(H(v),B_{H(v)})$.

If there is an infinite path $p$ on $E\backslash(H,B_{H})$ such that
$(E\backslash(H,B_{H}))^{0}=M(p)$, consider $p$ as an infinite path in the
graph $E$. Then $H=H(p)$ and, since $E\backslash(H,B_{H})$ has Condition (L),
$p$ cannot end in an exclusive cycle in $E$. Hence, it follows from Lemma
\ref{lem:annih-infpath}(1) that the annihilator of $V_{[p]}$ is precisely
$P=I(H,B_{H})$.

This concludes the proof of the theorem.
\end{proof}

From the proof of the above theorem dealing with condition (iii), we get the
following sharper reformulation of \cite[Theorem 5.7]{ABR},
characterizing primitive Leavitt path algebras.

\begin{theorem}
(\cite{ABR}) Let $E$ be an arbitrary graph and $K$ be any field. Then the
Leavitt path algebra $L_{K}(E)$ is primitive if and only if $E$ contains
either a sink $w$ or an infinite path $p$ that does not end in a cycle without exits, 
such that $E^{0}=M(w)$ or $M(p)$.
\end{theorem}

Given a graph $E$ and a hereditary subset $H$, one can construct the
\textit{restricted graph }$E_{H}$ as follows:%

\[
(E_{H})^{0}=H\text{ and }(E_{H})^{1}=\{e\in E^{1}:s(e)\in H\}.
\]

We now state a result on Chen modules that will be useful in next section.

\begin{lemma}
\label{lem:Morita-inv-Chenmods} Let $E$ be an arbitrary graph and let $H$ be a
hereditary subset of $E^{0}$. Consider the idempotent $e=\sum_{w\in H}w$ in
the multiplier algebra $M(L_{K}(E))$ of $L_{K}(E)$. Then $L_{K}(E_{H})$ is
isomorphic to $eL_{K}(E)e$. Moreover, if $M$ is a Chen simple $L_{K}(E_{H}%
)$-module of one the types (1)--(5) described in Definition
\ref{def:Chen-module}, then $L_{K}(E)e\otimes_{eL_{K}(E)e}M$ is also a Chen
simple $L_{K}(E)$-module of the same type.
\end{lemma}

\begin{proof}
The algebra $eL_{K}(E)e$ is linearly generated by all the paths $pq^{\ast}$,
where $p,q$ are paths in $E$ such that $s(p),s(q)\in H$. Using this, it is
easy to show that there is a graded isomorphism $L_{K}(E_{H})\rightarrow
eL_{K}(E)e$, which sends the vertices in $E_{H}^{0}=H$ to the same vertices in
$E^{0}$ and the edges in $(E_{H})^{1}$ to the corresponding edges in $E$.

Write $R=L_{K}(E)$ and $S=L_{K}(E_{H})$. If $M=V_{[p]}^{E_{H}}$ is a simple
$L_{K}(E_{H})$-module given by an infinite path $p$ on $E_{H}$, then we have
an isomorphism
\[
Re\otimes_{eRe}V_{[p]}^{E_{H}}\cong V_{[p]}%
\]
sending $\gamma\eta^{\ast}\otimes q$ to $\gamma\eta^{\ast}q$, where $\gamma$
and $\eta$ are finite paths in $E$ such that $s(\eta)\in H$, and $q$ is an
infinite path tail-equivalent to $p$ starting at a vertex of $H$. It is easily
checked that this map is surjective. To see that it is injective, observe that
we can write any element in $Re\otimes_{eRe}V_{[p]}^{E_{H}}$ in the form
$\sum_{i=1}^{r}\gamma_{i}\otimes a_{i}$, where $a_{i}\in V_{[p]}^{E_{H}}$, and
$\gamma_{i}$ are distinct paths in $E$ such that $r(\gamma_{i})\in H$ and all
the other vertices of the path $\gamma_{i}$ do not belong to $H$. If
$\sum_{i=1}^{r}\gamma_{i}a_{i}=0$, then we get $a_{i}=0$ for all $i$, and so
$\sum_{i=1}^{r}\gamma_{i}\otimes a_{i}=0$.

It is a simple matter to show that the same holds for all types of Chen simple
modules considered in Definition \ref{def:Chen-module}.
\end{proof}

\section{When every simple module is finitely presented}

Let $E$ be a finite graph. We first show that the simple module $V_{[p]}$ is
finitely presented if and only if $p$ is tail-equivalent to the rational
infinite path $ggg\cdot\cdot\cdot$ for some closed path $g$. We then show
that every (Chen) simple left $L:= L_{K}(E)$-module is finitely presented
if and only if every vertex in the graph $E$ is the base of at most one cycle
in $E$ (Theorem \ref{Simple Fin Presented}). Finally, we obtain a proof of our main result, 
Theorem \ref{theor:all-characterizations}.

The next proposition gives necessary and sufficient conditions under which the
simple left module $V_{[p]}$ (where $p$ is an infinite path) is finitely
presented. In its proof, we use the notation $P[\bar{E}]$ to denote the path
algebra over the field $K$ of the reverse graph $\bar{E}$ of $E$, with $\bar{E}^0= E^0$ and
$\bar{E}^1=\{ e^* : e\in E^1 \}$. 

A closed path $c$ is said to be {\it primitive} in case $c\ne d^r$ for any closed path $d$ and any $r\ge 2$.
We remark that these paths are called simple closed paths in \cite{C}.  
Observe that, if $c=e_1e_2\cdots e_n$ is primitive, then all the rotates 
$c_i = e_ie_{i+1} \cdots e_ne_1\cdots e_{i-1}$, $i=1,\dots , n$,  are different.

\begin{proposition}
\label{[p]Simply Presented} Let $E$ be a finite graph and $K$ be any field. Let
$p$ be an infinite path in $E$. Then the simple left $L_{K}(E)$-module
$V_{[p]}$ is finitely presented if and only if $p$ is tail-equivalent to the
(rational) infinite path $ccc\cdot\cdot\cdot$ where $c$ is some closed
path in $E$.
\end{proposition}

\begin{proof}
Suppose $V_{[p]}$ is a finitely presented left $L_{K}(E)$-module. First
observe that the simple $L_{K}(E)$- module $V_{[p]}$ can not be projective. This is
because, since the graph $E$ is finite, the infinite path $p$ can not be tail
equivalent to a path containing a line point. Since $V_{[p]}$ is simple, it is
of finite length and contains no projective submodule and so by Proposition
7.2 (1) of \cite{AB}, $V_{[p]}$ is a finitely generated Blanchfield module
over the path algebra $P[E]$. Also, by Proposition 7.2 (2) of \cite{AB},
$V_{[p]}=L_{K}(E)\otimes_{P[\bar{E}]}N$ for some $P[\bar{E}]$-module $N$
having finite $K$-dimension. Indeed we can assume $V_{[p]}=P[E]N$ for some
$P[\bar{E}]$-module $N$ which is finite dimensional over $K$. In particular, there
are infinite paths $q_{1},\cdot\cdot\cdot,q_{r}\in\lbrack p]$ such that every
element $a$ in $V_{[p]}$ can be written as $a=%
{\textstyle\sum\limits_{i=1}^{n}}
a_{i}q_{i}$ for some $a_{i}\in P[E]$. Since each $q_{i}\sim p$, we can assume
that $q_{i}=\tau_{>n_{i}}(p)$ for some positive integer $n_{i}$. Choose an
integer $m$ larger than $\max\{n_{1},\cdot\cdot\cdot,n_{r}\}$. By
hypothesis,
\[
\tau_{>m}(p)=%
{\textstyle\sum\limits_{i=1}^{r}}
a_{i}\tau_{>n_{i}}(p)
\]
where $a_{i}\in P[E]$. Write
\[
p=e_{1}e_{2}e_{3}\cdot\cdot\cdot\text{ ,}%
\]
where $e_{i}$ are the edges of the infinite path $p$. Observe that
\[
\tau_{>m+1}(p)=e_{m+1}^{\ast}\tau_{>m}(p)=%
{\textstyle\sum\limits_{i=1}^{r}}
e_{m+1}^{\ast}a_{i}\tau_{>n_{i}}(p)\text{.}%
\]
In this way, we can reduce the length of the paths in each of the terms
$a_{i}$, and also possibly obtain some terms of the form $\tau_{>n_{i}+1}(p)$
for some $i$. Repeating this process a finite number of times we arrive at the
equation
\[
\tau_{>k}(p)=%
{\textstyle\sum\limits_{i=1}^{r}}
\lambda_{i}\tau_{>k_{i}}(p)
\]
where $\lambda_{i}\in K$ and $k>\max\{k_{1},\cdot\cdot\cdot k_{r}\}$. This
implies that $\tau_{>k}(p)=\tau_{>l}(p)$ for some integer $l<k$. Thus we get
the equality of the infinite paths%
\[
e_{k+1}e_{k+2}\cdot\cdot\cdot=e_{l+1}e_{l+2}\cdots
\]
and so $e_{k+1}=e_{l+1},e_{k+2}=e_{l+2}\cdots,e_{2k-l}=e_{k}
,\, e_{2k-l+1}=e_{k+1}=e_{l+1},
\, e_{2k-l+2}=e_{k+2}=e_{l+2}\cdots$. We conclude that
\[
p=(e_{1}e_{2}\cdots e_{l})ccc\cdots
\]
where $c=e_{l+1}e_{l+2}\cdots e_{k}$ so that $p$ is tail-equivalent
to $c^{\infty}=$ $cccc\cdots$ , as desired.

Conversely, suppose $p$ is tail-equivalent to the infinite path $c^{\infty
}=ccc\cdots$ for some primitive path $c$ of length $n$, say
$c=e_{1}\cdots e_{n}$. Let $c_{1}=c$ and for each $i=2,...,n$, let
$c_{i}=e_{i}e_{i+1}\cdots e_{i-1}$ be the $i$-th rotate of $c$. For each $i$, let $p_{i}=c_{i}^{\infty}=c_{i}c_{i}c_{i}\cdot\cdot
\cdot$ be an infinite path. Then the finite dimensional $K$-vector space
$N=Kp_{1}\oplus\cdot\cdot\cdot\oplus Kp_{n}$ is actually a $P[\bar{E}]$-module
and, by Proposition 2.2 of \cite{AB}, $N$ is a finitely presented $P[\bar{E}]$-module. 
Now $V_{[p]}=L_{K}(E)N$ and for any sink $u\in E^{0}$, $u\cdot V_{[p]}=0$. Then by
Proposition 7.2 of \cite{AB}, $V_{[p]}$ is finitely presented as a left
$L_{K}(E)$-module.
\end{proof}

\begin{notation}
If $E$ is a graph and $v\in E^{0}$ is a source then $E\backslash v$ denotes
the "source elimination graph" where $(E\backslash v)^{0}=E^{0}\backslash
\{v\},(E\backslash v)^{1}=E^{1}\backslash s^{-1}(v)$, $s_{E\backslash
v}=s|(E\backslash v)^{1}$ and $r_{E\backslash v}=r|(E\backslash v)^{1}$
\end{notation}

The following lemma was proved in  \cite[Lemma 1.4]{ALPS} under the
assumption that $L_{K}(E)$ is simple and it can also be derived from \cite[Lemma 6.1]{ABC}. We
give a direct proof for completeness.

\begin{lemma}
\label{No sources}Let $E$ be a finite graph. If $v$ is a source and not a
sink, then $L_{K}(E)$ is Morita equivalent to $L_{K}(E\backslash v)$.
\end{lemma}

\begin{proof}
First, observe that the hypothesis that $v$ is a source but not a sink implies
that $|E^0|\geq2$.

Note that $E\backslash v$ is a complete subgraph of $E$. Hence, the
$K$-algebra map $\theta:L_{K}(E\backslash v)\longrightarrow L_{K}(E)$ given,
for all $w\in(E\backslash v)^{0}$, $e\in(E\backslash v)^{1}$, by
$\theta(w)=w,\theta(e)=e$ and $\theta(e^{\ast})=e^{\ast}$ is a non-zero graded
homomorphism. Since $\theta$ is non-zero at all the vertices of $E\backslash
v$, it is then a monomorphism.

Let $\epsilon=\theta(1_{L_{K}(E\backslash v)})=%
{\textstyle\sum\limits_{w\in E^{0},w\neq v}}
w$. We claim that $\theta(L_{K}(E\backslash v))=\epsilon L_{K}(E)\epsilon$.
Clearly $\theta(L_{K}(E\backslash v))\subseteq\epsilon L_{K}(E)\epsilon$. To
prove the other inclusion, note that $\epsilon L_{K}(E)\epsilon$ is linearly
spanned by elements $pq^{\ast}\in\epsilon L_{K}(E)\epsilon$ such that $s(p)\ne
v$ and $s(q)\ne v$. Moreover, since $v$ is a source $p$ as well as $q$ cannot
pass through $v$. Hence both $p$ and $q$ are paths in $E\backslash v$,
consequently $pq^{\ast}=\theta(pq^{\ast})\in\theta(L_{K}(E\backslash v))$,
thus proving our claim.

To show the Morita equivalence, we need also to show that $L_{K}(E)\epsilon
L_{K}(E)=L_{K}(E)$. It is enough to show that $v$ is in $L_{K}(E)\epsilon
L_{K}(E)$. Let $\{e_{1},\cdots ,e_{n}\}=s^{-1}(v)\neq\emptyset$. Since
$r(e_{i})$ belongs to the ideal $L_{K}(E)\epsilon L_{K}(E)$, the edge $e_{i}$
belongs to $L_{K}(E)\epsilon L_{K}(E)$, for all $i=1,\cdots ,n$. Then
$v=%
{\textstyle\sum\limits_{i=1}^{n}}
e_{i}e_{i}^{\ast}\in L_{K}(E)\epsilon L_{K}(E)$. This proves that
$L_{K}(E)\epsilon L_{K}(E)=L(E)$. Hence $L_{K}(E)$ is Morita equivalent to
$L_{K}(E\backslash v)$.
\end{proof}

\begin{lemma}
\label{cycle to loop}Let $E$ be a finite graph. Let $c$ be a cycle in $E$
without entries, that is, such that $|r^{-1}(v)|=1$ for all $v\in c^{0}$. Then
a finite graph $F$ can be constructed from $E$ in which the cycle $c$ is
replaced by a loop such that $L_{K}(F)$ is Morita equivalent to $L_{K}(E)$.
\end{lemma}

\begin{proof}
Write $c=e_{1}\cdots e_{r}$, with $s(e_{i})=v_{i}$ for all $i$. We
define a graph $F$ as follows:

Let $F^{0}=( E^{0}\backslash c^{0}) \cup\{v\}$ where $v$ is a new vertex.

To define $F^{1}$, note that by our hypothesis, $r(e)\notin c^{0}$ for all
$e\in E^{1}$ such that $s(e)\notin c^{0}$, that is, $E^{0}\setminus c^0 $ is a
hereditary set. So, we define $s_{F}^{-1}(w)=s_{E}^{-1}(w)$ if $w\in
E^{0}\setminus c^{0}$.

Corresponding to an edge $f$ with $s(f)\in c^{0}$ and $r(f)\in$ $E^{0}%
\backslash c^{0}$, define an edge $f^{\prime}$ in $F^{1}$ with $s_{F}%
(f^{\prime})=v$ and $r_{F}(f^{\prime})=r(f)$. Finally, we define a loop
$e^{\prime}$ at $v$ so that $s_{F}(e^{\prime})=v=r_{F}(e^{\prime})$.

We now define a map $\theta:L_{K}(F)\longrightarrow L_{K}(E)$ as follows:
$\theta(w)=w$ for all $w\in E^{0}\backslash c^{0}$ and $\theta(v)=v_{1}$
(where $v_{1}=s(e_{1})$). As for edges, $\theta(e)=e$ for all $e$ with
$s(e)\in$ $E^{0}\backslash c^{0}$.

We set $\theta(f^{\prime})=e_{1}\cdots e_{i-1}f$ if the edge $f$
corresponding to $f^{\prime}$ satisfies $s(f)=v_{i}$ and $r(f)\in
E^{0}\setminus c^0$.

Finally, set $\theta(e^{\prime})=$ $e_{1}\cdots e_{r}$.

It can be verified that the CK-relations are preserved under this map and so
$\theta$ extends to a well-defined algebra homomorphism from $L_{K}(F)$ onto
the corner $\epsilon L_{K}(E)\epsilon$ where $\epsilon=\theta(1_{L_{K}(F)})$.
(Note that $1_{L_{K}(F)}=v+
{\textstyle\sum\limits_{u\in E^{0}\backslash c^{0}}}
u$.) The most tricky part of this verification is to show the preservation of
(CK2) at vertex $v$. To show this, observe that
\begin{align*}
&  \sum_{\alpha\in s_{F}^{-1}(v)}\theta(\alpha\alpha^{\ast})=e_{1}\cdots
e_{r}e_{r}^{\ast}\cdots e_{1}^{\ast}+\sum_{i=1}^{r}\sum_{g\in s_{E}^{-1}%
(v_{i})\setminus\{e_{i}\}}e_{1}\cdots e_{i-1}gg^{\ast}e_{i-1}^{\ast}\cdots
e_{1}^{\ast}\\
&  =e_{1}\cdots e_{r-1}\Big(e_{r}e_{r}^{\ast}+\sum_{g\in s_{E}^{-1}%
(v_{r})\setminus\{e_{r}\}}gg^{\ast}\Big)e_{r-1}^{\ast}\cdots e_{1}^{\ast}\\
&  +\sum_{i=1}^{r-1}\sum_{g\in s_{E}^{-1}(v_{i})\setminus\{e_{i}\}}e_{1}\cdots
e_{i-1}gg^{\ast}e_{i-1}^{\ast}\cdots e_{1}^{\ast}\\
&  =e_{1}\cdots e_{r-1}e_{r-1}^{\ast}\cdots e_{1}^{\ast}+\sum_{i=1}^{r-1}%
\sum_{g\in s_{E}^{-1}(v_{i})\setminus\{e_{i}\}}e_{1}\cdots e_{i-1}gg^{\ast
}e_{i-1}^{\ast}\cdots e_{1}^{\ast}\\
&  =\cdots\,\,\cdots\,\,\cdots\\
&  =e_{1}e_{1}^{\ast}+\sum_{g\in s_{E}^{-1}(v_{1})\setminus\{e_{1}\}}gg^{\ast
}=v_{1}.
\end{align*}
Since $\theta(v)=v_{1}$, this shows that relation (CK2) at $v$ is preserved.

To show that the map $\theta$ is injective, we show that $\theta$ sends, in a
one-to-one way, a basis of $L_{K}(F)$ to a subset of a basis of $L_{K}(E)$. We
make use of the basis defined in \cite[Section 3]{AAJZ} (see also \cite[Chapter 1]{AAS}). 
For each $w\in E^{0}\setminus c^{0}$ which is not a sink, choose an edge $\gamma(w)$ in
$s_{E}^{-1}(w)$. For each $i=1,\dots,r$, set $\gamma(v_{i}) = e_{i}\in
s_{E}^{-1}(v_{i})$. Refer to these edges as special. By \cite[Theorem 1]%
{AAJZ}, a basis $\mathcal{B }_{E}$ of $L_{K}(E)$ is given by the following
elements (i) $w$, where $w\in E^{0}$, (ii) $p,p^{*}$, where $p$ is a path in
$E$, (iii) $pq^{*}$, where $p=e_{1}\cdots e_{n}$, $q=f_{1}\cdots f_{m}$ are
paths that end at the same vertex $r(e_{n})=r(f_{m})$, with $n,m\ge1$, with no
restriction when $e_{n}\ne f_{m}$, but with the restriction that $e_{n}$ must
not be special when $e_{n}=f_{m}$. In other words we avoid terms of the form
$e_{1}e_{2}\cdots e_{n}e_{n}^{*}f_{m-1}^{*}\cdots f_{1}^{*}$, with $e_{n}$ special.

Consider a corresponding basis $\mathcal{B }_{F}$ for $L_{K}(F)$ by declaring
as special the same edges $\gamma(w)$ as before for $w\in E^{0}\setminus
c^{0}$, and declaring $\gamma(v) = e^{\prime}$. Then it is clear that $\theta$
restricts to an injective mapping from the basis of $L_{K}(F)$ into the basis
of $L_{K}(E)$.

We now check that the image of $\theta$ is exactly $\epsilon L_{K}(E)\epsilon
$. Indeed, it is easy to check, using the hypothesis that $c$ has no entries,
that the subset $\theta(\mathcal{B }_{F})$ of the basis $\mathcal{B}_{E}$ of
$L_{K}(E)$ is a linear basis for $\epsilon L_{K}(E)\epsilon$, and so
$\theta(L_{K}(F))= \epsilon L_{K}(E)\epsilon$.

To show that $L_{K}(E)\epsilon L_{K}(E)=L_{K}(E)$, it is enough to observe
that $c^{0}\subseteq L_{K}(E)\epsilon L_{K}(E)$. This follows from the fact
that $v_{1}\in L_{K}(E)\epsilon L_{K}(E)$ and the equality $v_{i}=e_{i-1}%
^{*}\cdots e_{1}^{*}v_{1}e_{1}\cdots e_{i-1}$ for $i=2,\dots,r$.

We have shown that $L_{K}(F)$ is isomorphic to a full corner $\epsilon
L_{K}(E)\epsilon$ of $L_{K}(E)$, and so $L_{K}(E)$ is Morita equivalent to
$L_{K}(F)$.
\end{proof}

We introduce a pre-order $\leq$ on the set of cycles of a directed graph $E$,
as follows. If $c_{1}$ and $c_{2}$ are two cycles in $E$, set $c_{1}\leq
c_{2}$ in case there is a path from a vertex of $c_{2}$ to a vertex of $c_{1}%
$. Note that this is indeed a partial order in case $E$ satisfies the
graph-theoretic condition (ii) in Theorem \ref{Simple Fin Presented}. We say
that a cycle $c$ of $E$ is a \textit{maximal cycle} in case, for any cycle
$c^{\prime}$ in $E$, $c\leq c^{\prime}$ implies $c^{\prime}\leq c$.

\begin{theorem}
\label{Simple Fin Presented}Let $E$ be a finite graph and $K$ be any field.
Then the following conditions are equivalent for the Leavitt path algebra
$L=L_{K}(E)$:

(i) \ \ Every simple left $L$-module is finitely presented;

(ii) \ Every  Chen simple module is finitely presented.

(iii) \ Every vertex $v$ in $E$ is the base of at most one cycle.
\end{theorem}

\begin{proof} (i) $\implies $ (ii) is immediate.

Assume (ii) so that every simple Chen $L$-module is finitely presented. 
Assume, by way of contradiction, that there is a $v\in
E^{0}$ which is the base of two different cycles $g,h$. Consider the infinite path
\[
p=gh^{2}gh^{3}\cdots gh^{n}gh^{n+1}\cdots \text{ .}%
\]
This path $p$ cannot be tail-equivalent to the rational path $c^{\infty
}=ccc\cdots $ for any closed path $c$ in $E$. Hence, by Proposition
\ref{[p]Simply Presented}, the Chen simple module $V_{[p]}$ is not finitely
presented, a contradiction. Thus every vertex in $E$ is the base of at most
one cycle.

Assume (iii) so that every vertex in $E$ is the base of at most one cycle. We
need to show that every simple left $L$-module is finitely presented.

Let $n$ be the number of distinct cycles in $E$. We apply induction on $n$ to
show that every simple left $L$-module is finitely presented.

The base case is the case $n=0$, that is, the case where $E$ contains no
cycles. In that case, $L$ is a semi-simple artinian ring and all its
left/right simple modules are projective and hence finitely presented.

Suppose $n\ge1$ and that the result is true for graphs containing less than
$n$ cycles.

Using Lemma \ref{No sources} a finite number of times, we get a finite graph
without sources $F$, containing the same closed paths as $E$, and such that
$L_{K}(E)$ is Morita equivalent to $L_{K}(F)\times K^{t}$ for some $t\ge0$.
Since Morita equivalence between module categories (over unital rings)
preserves simple modules and finite presentation (see \cite{AF}), we can
therefore assume that the graph $E$ has no sources. Then all paths in $E$ can
be seen as portions of paths coming from a cycle in $E$.

Let $c$ be a maximal cycle. Since we are assuming that $E$ does not have
sources, the cycle $c$ has no entries. So, using Lemma \ref{cycle to loop}, we
may assume that $c$ is a loop, with $s(c)=v=r(c)$. Since $c$ is a maximal
cycle and $E$ has no sources, $E^{0}\backslash\{v\}$ is a hereditary saturated
set. Let $M$ be the (graded) ideal of $L$ generated by $E^{0}\backslash\{v\}$.
Clearly $L/M\cong K[x,x^{-1}]$.

Consider an arbitrary simple left $L$-module $S$.

Suppose $MS=S$. Let $e=\sum_{w\in E^{0}\setminus\{v\}}w$. Then $e$ is a full
idempotent in $M$, that is $LeL=MeM=M$, and so $M$ is Morita equivalent to
$eLe=L(E_{H})$, where $H:=E^{0}\setminus\{v\}$, and $E_{H}$ denotes the
restriction of $E$ to $H$, that is, the graph with $(E_{H})^{0}=H$ and
$(E_{H})^{1}=\{e\in E^{1}\mid s(e)\in H\}$. Note that we have a surjective
Morita context given by $(Le,eL)$, that is, we have surjective bimodule
homomorphisms
\[
eL\otimes_{L}Le\rightarrow eLe,\qquad Le\otimes_{eLe}eL\rightarrow LeL=M.
\]
Now $E_{H}$ contains $n-1$ cycles and so, by induction hypothesis, all simple
$eLe$-modules are finitely presented. Since $M$ is also a Leavitt path algebra
(see \cite[Lemma 1.2]{AP}), it has local units, and so we can apply 
\cite[Theorem 2.2]{AM}
to deduce that there is an equivalence of categories between $M$-Mod and
$eLe$-Mod, induced by the functors $eL\otimes_{M}-$ and $Le\otimes_{eLe}-$.
Therefore there is a simple $eLe$-module $S^{\prime}$ such that $Le\otimes
_{eLe}S^{\prime}\cong S$. \ Since $S^{\prime}$ is finitely presented, so is
$S$. Indeed if
\[
(eLe)^{n}\longrightarrow(eLe)^{m}\longrightarrow S^{\prime}\longrightarrow0
\]
is an exact sequence of $eLe$-modules, then
\[
(Le)^{n}\longrightarrow(Le)^{m}\longrightarrow Le\otimes_{eLe}S^{\prime
}\longrightarrow0
\]
is an exact sequence of $L$-modules, which shows that $S$ is finitely presented.

Suppose now that $MS=0$. Then $S$ is a simple $K[x,x^{-1}]$-module and so
there is a polynomial $f(x)=1+a_{1}x+\cdots +a_{n}x^{n}$, with
$a_{n}\neq0$ such that $S\cong K[x,x^{-1}]/(K[x,x^{-1}]f(x))$. Consequently,
$S\cong L/(Lf(c)+M)$.

Set $s^{-1}(v)\setminus\{c\}=\{e_{1},\dots,e_{k}\}$.

Let $N$ be the (finitely generated) left ideal of $L$ generated by
$H=E^{0}\setminus\{v\}$ and by all the paths of the form $e_{i}^{\ast}%
(c^{\ast})^{j}$, with $1\leq i\leq k$, $0\leq j\leq n-1$. Observe that
$e_{i}^{\ast}\in M$, so that $N\subseteq M$, as $M$ is an ideal of $L$. We claim that $Lf(c)+M=Lf(c)+N$.
This will show that $S$ is finitely presented. Note that $M$ is linearly spanned
by the elements of the form $pq^{\ast}$, where $p,q$ are paths in $E$ such
that $r(p)=r(q)\in H$. If $r(pq^{\ast})=s(q)\in H$ then $pq^{\ast}\in
Ls(q)\subseteq N$. Therefore we can assume that $s(q)=v$. In this case observe
that $q=c^{j}e_{i}q_{1}$ for some $1\leq i\leq k$ and $j\geq0$, and some path
$q_{1}$. Therefore $pq^{\ast}\in Le_{i}^{\ast}(c^{\ast})^{j}$. It thus
suffices to show that $e_{i}^{\ast}(c^{\ast})^{j}$ belongs to $Lf(c)+N$ for
all $i,j$. If $0\leq j\leq n-1$, this follows from the definition. Suppose
that $e_{i}^{\ast}(c^{\ast})^{t}\in Lf(c)+N$ for all $0\leq t\leq r$, where
$r\geq n-1$. Then multiplying $f(c)$ on the left by $e_{i}^{\ast}(c^{\ast
})^{r+1}$ we obtain
\[
e_{i}^{\ast}(c^{\ast})^{r+1}=e_{i}^{\ast}(c^{\ast})^{r+1}f(c)-a_{1}e_{i}%
^{\ast}(c^{\ast})^{r}-\cdots-a_{n}e_{i}^{\ast}(c^{\ast})^{r+1-n}\in Lf(c)+N\,.
\]
Thus $S=L/(Lf(c)+N)$ is finitely presented, as desired.
\end{proof}

For the class of graphs $E$ appearing in Theorem \ref{Simple Fin Presented},
we can indeed classify all the simple left $L_{K}(E)$-modules. Specifically,
we show that in this case every simple $L_{K}(E)$-module determines and is
determined by a unique primitive ideal of $L_{K}(E)$.

\begin{corollary}
\label{simple is Chen}Let $E$ be a finite graph such that every vertex in $E$
is the base of at most one cycle. Then every simple $L_{K}(E)$-module is a
Chen module. Indeed, for any primitive ideal $P$ of $L_{K}(E)$ there exists a
unique simple $L_{K}(E)$-module $S$ (which is a Chen module) such that the
annihilator of $S$ is $P$.
\end{corollary}

\begin{proof}
The proof uses the same kind of induction as in Theorem
\ref{Simple Fin Presented}. Let $n$ be the number of distinct cycles in $E$.
If $n=0$, then $E$ is semisimple artinian, and the simple modules are in
bijective correspondence with the sinks of $E$. So all of them are of the form
$\mathbf{N}_{w}$ for a sink $w$, and distinct simple modules have distinct
annihilators. Assume the result is true for graphs with less than $n$ distinct
cycles, and let $E$ be a finite graph with $n$ cycles satisfying the required
hypothesis. Since, by Theorem \ref{thm:realizingprims}, we can realize every
primitive ideal as the annihilator of at least one Chen simple module, and
since, for two Morita-equivalent unital rings $R$ and $S$, there is a
bijective correspondence between the isomorphism classes of simple $R$-modules
and the isomorphism classes of simple $S$-modules, and also a bijective
correspondence between primitive ideals of $R$ and primitive ideals of $S$
which is compatible with the above in the sense that respects annihilators of
simple modules, we may use Lemma \ref{No sources} a finite number of times and
reduce to the case where $E$ does not have sources.

Let $c$ be a maximal cycle in $E$. Using Lemma \ref{cycle to loop}, we can
further assume that $c$ is a loop, based at $v$. Let $M=I(H)$, where
$H=E^{0}\setminus\{v\}$.

Set $L:=L_{K}(E)$. If $S$ is a simple $L$-module such that $MS=S$, then as in
the proof of Theorem \ref{Simple Fin Presented} we have that $S\cong
Le\otimes_{eLe}S^{\prime}$, where $S^{\prime}$ is a simple $L_{K}(E_{H}%
)$-module. By the inductive hypothesis, $S^{\prime}$ is a Chen $L_{K}(E_{H}%
)$-module, and therefore $S$ is a Chen $L_{K}(E)$-module by Lemma
\ref{lem:Morita-inv-Chenmods}. Moreover if $S_{1}$ and $S_{2}$ are two simple
$L$-modules such that $MS_{i}=S_{i}$ for $i=1,2$, and $S_{1}\ncong S_{2}$,
then $S_{1}^{\prime}\ncong S_{2}^{\prime}$ and so by induction hypothesis,
$\text{Ann}_{L(E_{H})}(S_{1}^{\prime})\neq\text{Ann}_{L(E_{H})}(S_{2}^{\prime
})$, which implies that
\begin{align*}
\text{Ann}_{M}(S_{1})  & =Le\otimes_{eLe}\text{Ann}_{L(E_{H})}(S_{1}^{\prime
})\otimes_{eLe}eL\\
& \neq Le\otimes_{eLe}\text{Ann}_{L(E_{H})}(S_{2}^{\prime})\otimes
_{eLe}eL=\text{Ann}_{M}(S_{2})
\end{align*}
and so $\text{Ann}_{L}(S_{1})\neq\text{Ann}_{L}(S_{2})$. If $MS_{1}=S_{1}$ and
$MS_{2}=0$ then $M\nsubseteq\text{Ann}_{L}(S_{1})$ and $M\subseteq
\text{Ann}_{L}(S_{2})$, so that $S_{1}$ and $S_{2}$ have different annihilators.

Finally suppose that $MS_{1}=0=MS_{2}$ and that $S_{1}\ncong S_{2}$. Then
there exist distinct irreducible polynomials $f(x)$ and $g(x)$ in
$K[x,x^{-1}]$ such that $S_{1}\cong L/(M+Lf(c))$ and $S_{2}\cong L/(M+Lg(c))$.
Therefore $\text{Ann}_{L} (S_{1}) = M+Lf(c) \ne M+Lg(c)=\text{Ann}_{L}
(S_{2})$. Also it can be easily verified that $S_{1}\cong V^{f}_{[q]}$, where
$q=c^{\infty}$, so that $S_{1}$ is a Chen simple module.

This completes the proof.
\end{proof}

\begin{example}
 \label{ex:Jacobson-alg}
 Let $E$ be the graph with $E^0=\{v,w \}$ and $E^1=\{e,f\}$ such that
$s(e)=r(e)=v=s(f)$ and $r(f)=w$. Then the Leavitt path
algebra $L_K(E)$ is the Jacobson algebra $\mathbb S _1 = K\langle x,y\mid yx=1 \rangle .$
This algebra is non-noetherian but all the simple modules are finitely presented by
Theorem \ref{Simple Fin Presented}. The structure of the simple $L_K(E)$-modules is well-known
(see \cite[Lemma 3.1]{bavula} or \cite[5.10(3)]{AB}). The 
Chen module $\mathbf{N}_w$ corresponding to the sink $w$ is the simple module
$K[x]$ (cf. \cite{bavula}). The other simple modules are the Chen modules $V^f_{[q]}$, where $q=e^{\infty}$
and $f(x)$ is an irreducible polynomial in $K[x,x^{-1}]$.
In \cite{bavula}, Bavula finds all the simples modules over the algebras 
$\mathbb S _n:= \mathbb S_1\otimes \overset{(n)}{\cdots} \otimes  \mathbb S_1$, for all $n\ge 1$.
It would be interesting to know whether similar results can be obtained for tensor products of Leavitt path algebras of the form $L_K(E)$,
where $E$ is a graph such that every vertex is the basis of at most one cycle.
\end{example}

We are now ready to prove our main result.

\medskip

\noindent {\it Proof of Theorem \ref{theor:all-characterizations}:}
 (1) $\iff $ (2) $\iff $ (3) is Theorem \ref{Simple Fin Presented}, and (3) $\implies $ (4) is shown in Corollary \ref{simple is Chen}.
 
 (4) $\implies $ (5). If $M$ is a simple $L$-module, then, by Theorem \ref{thm:realizingprims} there exists a Chen module $N$
 such that $\text{Ann}_L(M)=\text{Ann}_L(N)$. Now (4) gives $M\cong N$. Therefore, every simple $L$-module is a Chen module.
 
 (5) $\implies $ (3). Suppose that $v\in E^0$ is the base of two different cycles $g$ and $h$. We will build a simple finitely presented  
 left $L$-module which is not a Chen module. 
 
 We begin by building a simple, finite-dimensional, $P[\ol{E}]$-module. We write $g=\alpha_1 \cdots \alpha _n$, with $v=v_1=s(g)=r(g)$ 
 and $v_i= s(\alpha_i) = r(\alpha_{i-1})$ for $i=2,\dots ,n$. Similarly, set $h=\beta_1\cdots \beta _m$, with $w_1=v=s(h)=r(h)$ and 
  $w_j= s(\beta_j) = r(\beta_{j-1})$ for $j=2,\dots ,m$. For $i=1,\dots ,n $, put $M_{v_i} = z_iK$, a $1$-dimensional vector space.
  For $w_j\in h^0\setminus g^0$, set $M_{w_j}= t_jK$, a $1$-dimensional vector space. If $w_j=v_i\in g^0\cap h^0$, set $t_j=z_i$. 
  Set also $M_w= 0$ if $w\notin g^0\cup h^0$. 
 This defines a family $(M_w)_{w\in E^0}$ of finite-dimensional vector spaces. Now, for $i=1,\dots ,n$, define a linear map 
 $$\Phi_{\alpha_i^*}\colon M_{v_{i+1}}\longrightarrow M_{v_i}$$
 by $\Phi_{\alpha_i^*}(z_{i+1})= z_i$ (where $v_{n+1}= v_1$). Similarly, if $j=1,\dots , m$ let $\Phi_{\beta_j^*}\colon M_{w_{j+1}}\to M_{w_j}$
 be the linear such that $\Phi _{\beta_j^*}(t_{j+1}) = t_j$. If $\alpha \notin g^1\cup h^1$, define $\Phi_{\alpha^*}= 0$. In this way, we have
 defined a family $\Big( (M_w)_{w\in E^0}, (\Phi _{\alpha^*})_{\alpha\in E^1} \Big)$, which gives rise to a $P[\ol{E}]$-module $M$, with underlying vector space
 $M=\bigoplus _{w\in E^0} M_w$. Observe that $\text{dim}_K (M)= | g^0 \cup h^0 |$.
 
We claim that $M$ is a simple $P[\ol{E}]$-module. To see this, let 
$$a= \sum_{i= 1}^n \lambda_i z_i + \sum _{j\in J} \mu _j  t_j$$
be a nonzero element in $M$, where $J= \{j\in \{1,\dots , m\} : w_j\notin g^0 \}$. Assume that $\lambda _{i_0}\ne 0$. 
Then $\alpha_{i_0-1}^* a = \lambda_{i_0} z_{i_0-1}\ne 0$ (where $i_0-1$ is computed mod $n$). Similarly, if $\mu_{j_0}\ne 0$ for some $j\in J$, 
then $\beta_{j_0-1}^* a= \mu_{j_0} w_{j_0-1}\ne 0$. In either case we obtain that $z_1=t_1\in P[\ol{M}]a$, which implies the simplicity of $M$. 

By \cite[Lemma 5.7]{AB},  $\widetilde{M}:=  L\otimes _{P[\ol{E}]} M$ is a finitely presented simple $L$-module.
Assume, by way of contradiction, that $\widetilde{M} $ is a Chen module. It is easy to show, by using the arguments in Lemma \ref{lem:annih-infpath}(1) 
that the annihilator of $\widetilde{M}$
is $I(H(g^{\infty}))=I(H(h^{\infty}))$, and so $\widetilde{M}$ cannot be a Chen module of type (1) or (5). Hence, there is an infinite path $p$
such that $\widetilde{M}\cong V_{[p]}$. 
Now, it follows from Proposition \ref{[p]Simply Presented} 
that $p$ must be tail-equivalent to a rational path of the form $q^{\infty}$, where $q$
is a primitive closed path in $E$, so that $\widetilde{M}\cong V_{[q^{\infty}]}$. 
Write $q=  e_1\cdots e_r$, with $e_i\in E^1$. For $i=1,\dots ,r $, let $q_i=e_i \cdots e_{i-1}$ be the $i$-th rotate of $q$.
Then $N= \bigoplus _{i=1}^r q_i^{\infty} K$ is a simple, finite-dimensional $P[\ol{E}]$-module, and 
$P[E]N= V_{[q^{\infty}]}$. This implies that $N$ is the smallest lattice of $ V_{[q^{\infty}]}$, see \cite[Proposition 7.2(3)]{AB}. Since the minimal lattice of $\widetilde{M}$ is $M$, we obtain 
a $P[\ol{E}]$-isomorphism $\phi \colon M\to N$. In particular $r= \text{dim}_K (N)= |g^0\cup h^0 |$. Moreover $\text{dim}_K( w N) = 1$ for all $w\in g^0\cup h^0$, which implies that 
$q$ must be a cycle, because it cannot pass through the same vertex twice. Let $i$ be the smallest positive integer such that $\alpha _i \ne \beta _i$. 
Then either $\alpha_i^* q_i^{\infty} = 0$ or $\beta_i^*q_i^{\infty} = 0$. Assume, for convenience, that $\alpha_i^* q_i^{\infty} =0$. Then, since $q$ is a cycle, we must have
$\alpha_i ^*q_j^{\infty}= 0$ for all $j=1,\dots , n$, and thus $\alpha_i^*N=0$. Hence $0= \alpha_i ^*M\ne 0$ and we have arrived to a contradiction. 
Therefore, we conclude that $\widetilde{M}$ is not a Chen module. \qed

\begin{remark}
From Theorem 5 of \cite{AAJZ}, it is interesting to observe that, for a finite graph $E$, 
the equivalent conditions of Theorem \ref{theor:all-characterizations}
are also equivalent to the condition that 
$L_{K}(E)$ has finite Gelfand-Kirillov dimension.
\end{remark}

\textbf{Acknowledgement:} 
Initial part of this work was done when the second-named author visited the
Universitat Autonoma de Barcelona during May 2013 and he gratefully
acknowledges the support and the hospitality of the faculty of the Department
of Mathematics.

\end{document}